\documentclass{amsart}

\usepackage{multicol}
\usepackage{colortbl}
\usepackage[dvips]{graphics}
\usepackage{exscale}
\usepackage{epsfig}
\usepackage{pst-plot,pst-infixplot,pstricks,graphicx}
\usepackage{amssymb,latexsym,amsthm,amsfonts,color,fancyhdr,mathrsfs}
\usepackage{lscape}
\usepackage{rotating}
\usepackage{caption}

\newtheorem{theorem}{Theorem}[section]
\newtheorem{proposition}[theorem]{Proposition}
\newtheorem{lemma}[theorem]{Lemma}
\newtheorem{corollary}[theorem]{Corollary}

\theoremstyle{definition}

\theoremstyle{remark}
\newtheorem{remark}[theorem]{Remark}

\numberwithin{equation}{section}

\begin{document}

\title[$H_q-$Semiclassical OP via polynomial mappings]
{$H_q-$Semiclassical orthogonal polynomials via polynomial
mappings}

%\author{}

\author{K. Castillo}
\address{(KC) CMUC, Department of Mathematics, University of Coimbra,
3001-501 Coimbra, Portugal}
\email{kenier@mat.uc.pt}
\author{M. N. de Jesus}
\address{(MNJ) Departamento de Matem\'atica, Escola Superior de Tecnologia e Gest\~ao,
Campus Polit\'ecnico de Repeses, 3504-510 Viseu, PORTUGAL}
\email{mnasce@estv.ipv.pt}
\author{F. Marcell\'an}
\address{(FM) Instituto de Ciencias Matem\'aticas (ICMAT), Calle Nicol\'as Cabrera 13-15, Campus de Cantoblanco UAM, 28049 Madrid, Spain,  and Departamento de Matem\'aticas, Universidad Carlos III de Madrid,
Avenida de la Universidad 30, 28911, Legan\'es, Spain}
\email{pacomarc@ing.uc3m.es}
\author{J. Petronilho}
\address{(JP) CMUC, Department of Mathematics, University of Coimbra,
3001-454 Coimbra, Portugal}
\email{josep@mat.uc.pt}


\date{\today}


\begin{abstract}
In this work we study orthogonal polynomials via polynomial mappings in the framework of the $H_q-$semiclassical class.
We consider two monic orthogonal polynomial sequences $\{p_n (x)\}_{n\geq0}$ and $\{q_n(x)\}_{n\geq0}$ such that
$$
p_{kn}(x)=q_n(x^k)\;,\quad n=0,1,2,\ldots\;,
$$
being $k$ a fixed integer number such that $k\geq2$, and we prove that if one of the
sequences $\{p_n (x)\}_{n\geq0}$ or $\{q_n(x)\}_{n\geq0}$ is
$H_q-$semiclassical, then so is the other one. In particular, we show that if
$\{p_n(x)\}_{n\geq0}$ is $H_q-$semiclassical of class $s\leq k-1$, then $\{q_n (x)\}_{n\geq0}$ is $H_{q^k}-$classical.
This fact allows us to recover and extend recent results in the framework of cubic transformations,
whenever we consider the above equality with $k=3$.
The idea of blocks of recurrence relations introduced by Charris and Ismail plays a key role in our study.
\end{abstract}
\maketitle

\section{Introduction}

Polynomial mappings constitute an interesting topic in the theory of orthogonal polynomials since their relations with Julia sets and almost periodic Jacobi matrices. Given a sequence of orthogonal polynomials with respect to a probability measure $\mu$ supported on a set $I\subseteq [-1,1]$, polynomial mappings provide a general approach to the analysis of polynomials orthogonal with respect to a measure defined by a polynomial transformation such that the inverse of $I$ is a real set that, in general, will be the union of a finite number of intervals such that any two of these intervals have at most one common point. These polynomials appear in the study of sieved orthogonal polynomials by using blocks of recurrence relations, see \cite{CharrisIsmail} and \cite{CharrisIsmailMonsalve}.  Applications of polynomial mappings in quantum chemistry and solid state physics can be find in \cite{Wheeler} and \cite{Pettifor}, respectively. A general framework is provided in \cite{JerVan} and the updated related works \cite{Peherstorfer} and \cite{MarcioPetronilhoJAT}.

The case of quadratic mappings has been studied by many authors. In particular, in \cite{ChiharaBUMI} the following problem is solved: Given a sequence of orthogonal polynomials $\{p_{n}(x)\}_{n\geq0}$ to find a symmetric sequence of orthogonal polynomials $\{q_{n}(x)\}_{n\geq0}$  such that  $q_{2n}(x)= p_{n}(x^{2}).$ In this case, $q_{2n+1}(x)= K_{n}(x^{2}),$ where $K_{n}(x)$ is the so called kernel polynomial of degree $n$ (see \cite{Chihara}). Later one, in \cite{LChihara} a quite general problem concerning the orthogonality of sequences of polynomials  $\{R_{n}(x)\}_{n\geq0}$ defined by $R_{2n}(x)= p_{n}(x^{2})+ \theta_{2n} x K_{n-1}(x^{2}), R_{2n+1}(x)= x K_{n}(x^{2})+ \theta_{2n+1} p_{n}(x^{2}), n\geq0,$ is analyzed. Indeed, necessary and sufficient conditions for such an orthogonality are deduced. A more general situation is described in \cite{Maroni1}, where the study of general quadratic decompositions of sequences of monic orthogonal polynomials $\{B_{n}(x)\}_{n\geq0}$ such that  $B_{2n}(x)= p_{n}(x^{2})+ x a_{n-1}(x^{2}), B_{2n+1}(x)= x R_{n}(x^{2})+  b_{n}(x^{2}), n\geq0,$ with $p_{n}(x), R_{n}(x)$ polynomials of degree $n$, and $a_{n}(x), b_{n}(x)$ polynomials of degree at most $n,$ is analyzed. Necessary and sufficient conditions for the orthogonality of the sequences of polynomials $\{p_{n}(x)\}_{n\geq0}$ and $\{R_{n}(x)\}_{n\geq0}$ are given. This idea of quadratic decomposition in a more general framework is the topic presented in \cite{Maroni2}. Finally, in \cite{PacoZeLAA,PacoZePortMath}, given a quadratic polynomial $\pi_{2}(x)$ the orthogonality of sequences of monic polynomials $\{B_{n}(x)\}_{n\geq0}$ such that either $B_{2n}(x)= p_{n}(\pi_{2}(x))$ or $B_{2n+1}(x)= (x-c) p_{n}(\pi_{2}(x))$ is studied and the relation between the corresponding linear functionals is obtained.

The study of the cubic case comes back to the pioneer work \cite{BarrucandDickinson} where assuming that  $\{p_{n}(x)\}_{n\geq0}$ is a symmetric sequence of monic orthogonal polynomials, then necessary and sufficient conditions for the orthogonality of a symmetric sequence of orthogonal polynomials $\{B_{n}(x)\}_{n\geq0}$ such that $B_{3n}(x)= p_{n}(x^{3}+ bx)$ are given. Such constrains (symmetry and the particular choice of the cubic polynomials) have been removed in \cite{PacoZeCubic1,PacoZeCubic2}, where the authors consider the problem of orthogonality of sequences $\{B_{n}(x)\}_{n\geq0}$ such that $B_{3n+m}(x)= \theta_{m}(x) p_{n}(\pi_{3}(x)), m\in\{0,1,2\}$, where $\pi_{3}(x)$ is a fixed cubic polynomial and $\theta_{m}(x)$ is a fixed polynomial of degree $m$. The problem of general cubic decompositions has been studied in \cite{Mesquita} following the hints of the quadratic case.

Nevertheless, questions related to cubic decompositions of orthogonal polynomial sequences satisfying some extra conditions as their semiclassical character have not been considered in the literature up to the recent contributions \cite{TounsiBouguerra} and \cite{TounsiRaddaoui} for particular cases of semiclassical and $H_{q}-$semiclassical orthogonal polynomials of class one. A more general framework concerning semiclassical orthogonal polynomials, including some particular polynomial mappings, is presented in \cite{KMP}.

The aim of the present  contribution is to analyze  sequences of monic orthogonal polynomials $\{p_{n}(x)\}_{n\geq0}$  and $\{q_{n}(x)\}_{n\geq0}$  such that $p_{nk}(x)= q_{n}(x^{k}), k\geq 2,$ and to study how the $H_{q}-$semiclassical character of the sequences is preserved. The novelty of  our results is related to the analysis of $H_{q}$-semiclassical orthogonal polynomials generated by such polynomial mappings by using the connection between the corresponding Stieltjes functions as a method to generate new examples of sequences of $H_{q}-$semiclassical orthogonal polynomials of class greater than or equal to 1, taking into account that the classification of such sequences, even for class 1, remains an open problem.

The structure of the manuscript is the following.
In Section 2 we present the basic background concerning $H_{q}-$semiclassical linear functionals as well as some properties about sequences of orthogonal polynomials defined by polynomial mappings. In Section 3 we deal with the stability, i.e. the preservation of the semiclassical character,  of  $H_{q}$-semiclassical linear functionals when polynomial mappings are introduced. The key idea is the consideration of the formal Stieltjes series associated with both linear functionals. In particular, the case of the polynomial mapping $\pi_{k}(x)=x^{k}, k\geq2$ is studied and the class of the associated linear functional is discussed  (Theorem 3.5). Finally, in Section 4, some illustrative computational examples of $H_{q}-$semiclassical sequences of orthogonal polynomials of classes 1 and 2 involving Little $q$-Laguerre and Little $q$-Jacobi polynomials are deeply studied, including discrete measure representations for some of the considered examples.

\section{Background}\label{Prel}

In this section we recall some basic facts concerning the general theory of orthogonal polynomials (OP)
that will be needed in the sequel.

\subsection{Basic definitions}

We denote by $\mathcal{P}$ the vector space of polynomials
with coefficients in $\mathbb{C}$ and by $\mathcal{P}^*$ its
dual space. The action of a functional
$\textbf{u}\in\mathcal{P}^*$ over a polynomial $f\in\mathcal{P}$ will be represented by
$\langle\textbf{u},f\rangle$. In particular,
$u_n:=\langle\textbf{u}, x^n\rangle$ is the moment of order $n$ of
$\textbf{u}$. In $\mathcal{P}^*$ we define the
$q-$derivative of a functional $\textbf{u}$ by
$$\left\langle H_q\textbf{u},f\right\rangle:=-\left\langle \textbf{u},H_q f\right\rangle\,$$
where $H_q$ is the Hahn's operator defined as
$$\left(H_q f\right)(x):=\frac{f(qx)-f(x)}{(q-1)x}\;,\quad f\in\mathcal{P}\;, x\neq0,$$ being
$q\in\widetilde{\mathbb{C}}:=\big\{ z\in\mathbb{C}\setminus\{0\} \;|\;
z^n\neq 1\;,\;\forall\; n\in\mathbb{N}\big\}$ \cite{Hahn,KM,Lotfi}.
In particular, this yields
$$\big(H_q\textbf{u})_n:=\langle H_q\textbf{u},x^n\rangle=-[n]_qu_{n-1}\;,\quad n\in\mathbb{N}_0\,,$$
where $[n]_q$ denotes the basic $q-$number, defined by
$$[0]_q:=0\;,\quad [n]_q:=1+q+\cdots+q^{n-1}\equiv\frac{q^n-1}{q-1}\;,\quad n\in\mathbb{N}\,.$$

Given $\textbf{u}\in\mathcal{P}^*$ and
$\phi\in\mathcal{P}$, the dilation of $\textbf{u}$
and the left-multiplication of a polynomial $\phi$ by $\textbf{u}$, are the
functionals $h_d\textbf{u},\phi \textbf{u}\in\mathcal{P}^*$ ($d\in\mathbb{C}\setminus\{0\}$)
defined by
$$\langle h_d\textbf{u},f\rangle:=\langle\textbf{u},h_d(f)\rangle=\langle\textbf{u},f(dx)\rangle\; ,\quad
\langle \phi \textbf{u},f\rangle:=\langle\textbf{u},\phi
f\rangle\; ,\quad f\in\mathcal{P}\, \;.
$$

Let $\textbf{u}\in\mathcal{P}^*$. A sequence $\{p_n(x)\}_{n\geq0}$ in $\mathcal{P}$
is said to be an orthogonal polynomial sequence (OPS) with respect to $\textbf{u}$
if the following two conditions hold:
\begin{enumerate}
\item[{\rm (i)}] $\deg p_n=n$ for each $n\in\mathbb{N}_0$;
\item[{\rm (ii)}] $\langle\textbf{u},p_np_m\rangle=k_n\delta_{n,m}$ for every $m,n\in\mathbb{N}_0$,
\end{enumerate}
where $\{k_n\}_{n\geq0}$ is a sequence of nonzero complex numbers and $\delta_{n,m}$ is the Kronecker symbol.
Under these conditions we say that $\textbf{u}$ is regular (or quasi-definite) \cite{Chihara}.

\subsection{$H_q-$Semiclassical OPS}

A linear functional ${\bf u}\in\mathcal{P}^*$ is called
$H_q-$semiclassical if it is regular and there exist two nonzero
polynomials $\Phi$ and $\Psi$ such that
\begin{equation}\label{grauPsi}
\mbox{\rm deg}\,\Psi\geq1
\end{equation}
and ${\bf u}$ satisfies the  functional equation
\begin{equation}\label{EqDistSC1}
H_q(\Phi{\bf u})=\Psi{\bf u}\;.
\end{equation}
If $\{p_n(x)\}_{n\geq0}$ is an OPS with respect to a
$H_q-$semiclassical functional then $\{p_n (x)\}_{n\geq0}$ is called a
$H_q-$semiclassical OPS. We point out the following useful
criterion.

\begin{proposition}\label{PropositionPhiPsi}
Let ${\bf u}\in\mathcal{P}^*$ be a  regular linear functional. Then,
${\bf u}$ is $H_q-$semiclassical if and only if there exist two
polynomials $\Phi$ and $\Psi$, with at least one of them nonzero,
such that $(\ref{EqDistSC1})$ holds. Moreover, under these
conditions, necessarily both $\Phi$ and $\Psi$ are nonzero and
$\Psi$ satisfies $(\ref{grauPsi})$.
\end{proposition}

If ${\bf u}$ is a $H_q-$semiclassical functional,
denoting by $\mathcal{A}_{\bf u}$ the set of all pairs $(\Phi,\Psi)$
of nonzero polynomials such that (\ref{grauPsi})--(\ref{EqDistSC1}) holds,
the nonnegative integer number
\begin{equation}\label{def-classSC}
s:=\min_{(\Phi,\Psi)\in\mathcal{A}_{\bf u}}\max\{\mbox{\rm
deg}\,\Phi-2,\mbox{\rm deg}\,\Psi-1\}
\end{equation}
is called the class of ${\bf u}$. The pair
$(\Phi,\Psi)\in\mathcal{A}_{\bf u}$ where the class of ${\bf u}$
is attained is unique. If $\{p_n\}_{n\geq0}$ is an OPS with
respect to a $H_q-$semiclassical functional of class $s$ then
$\{p_n(x)\}_{n\geq0}$ is called a $H_q-$semiclassical OPS of class
$s$. In particular, when $s=0$ (so that $\mbox{\rm
deg}\,\Phi\leq2$ and $\mbox{\rm deg}\,\Psi=1$) $\{p_n(x)\}_{n\geq0}$
is called a $H_q-$classical OPS.

Table 1 summarizes the two canonical forms for the
pairs $(\Phi,\Psi)$ corresponding to the $H_q-$classical OPS
$\{L_n(x;a|q)\}_{n\geq0}$ and $\{U_n(x;a,b|q)\}_{n\geq0}$,
known as the little $q-$Laguerre polynomials
and the little $q-$Jacobi polynomials, respectively.
We only need to include these two sequences of $H_q-$classical OPS because they
are the only ones that will appear in the examples given at the last Section. For details, see \cite{Renato-libro,KM,MedemRenatoPaco}.

{\begin{table}
\scriptsize
\centering
\begin{tabular}{|>{\columncolor[gray]{0.95}}c|c|c|c|}
\hline \rowcolor[gray]{0.95}
\rule{0pt}{1.2em} $p_n(x)$ & $\Phi$ & $\Psi$ & regularity conditions \\
\hline
\rule{0pt}{1.2em} $L_n(x;a|q)$ & $x$& $a^{-1}q^{-1}(q-1)^{-1}(x-1+aq)$ &
$a\neq0$; $a\neq q^{-n-1}$, $n\geq0$ \\
%\hline
\rule{0pt}{1.2em} $U_n(x;a,b|q)$ & $x\big(x-b^{-1}q^{-1}\big)$ &
$a^{-1}b^{-1}q^{-2}(q-1)^{-1}\big(\big(abq^2-1\big)x+1-aq\big)$ &
$ab\neq0\,, q^{-n} $; $a,b\neq q^{-n-1}$, $n\geq0$ \\
\hline
\end{tabular}
\medskip
\caption{The pairs $(\Phi,\Psi)$ for the canonical forms corresponding to the
little $q-$Laguerre OPS and to the little $q-$Jacobi OPS.}\label{Table1}
\end{table}}

Among several very well known characterizations of  $H_q-$semiclassical OPS,
we recall the following one \cite[Proposition 3.1]{Lotfi}:
${\bf u}\in\mathcal{P}^*$ is $H_q-$semiclassical if and only if it is
regular and the associated Stieltjes formal series,
$$S_{\bf u}(z):=-\sum_{n=0}^\infty \frac{u_n}{z^{n+1}}\;,$$
satisfies (formally) the equation
\begin{equation}\label{StieltSC}
A(z)\left(H_{q^{-1}}S_{\bf u}\right)(z)=C(z)S_{\bf u}(z)+D(z)\;,
\end{equation}
where $A$, $B$, and $C$ are polynomials, $A$ being nonzero.
Moreover, if ${\bf u}$ satisfies (\ref{EqDistSC1}),
then the polynomials $A$, $C$ and
$D$ in (\ref{StieltSC}) are given  in terms of the polynomials $\Phi$ and $\Psi$ as follows
$$\begin{array}{l}
A(z)=q^{{\small\mbox{deg}}\Phi}\left(h_{q^{-1}}\Phi\right)(z)\;
,\quad
C(z)=q^{{\small\mbox{deg}}\Phi}\left(q\Psi(z)-\left(H_{q^{-1}}\Phi\right)(z)\right)\;
,\quad
\\D(z)=q^{{\small\mbox{deg}}\Phi}\left(q\left({\bf
u}\theta_0\Psi\right)(z)-\left(H_{q^{-1}}\big({\bf
u}\theta_0\Phi\big)\right)(z)\right)\;, \end{array}$$
where, for each $f\in\mathcal{P}$ and ${\bf u}\in\mathcal{P}'$,
$\theta_0f$ and ${\bf u}f$ are polynomials, defined by
$$
\theta_0f(x):=\frac{f(x)-f(0)}{x}\,,\quad
{\bf u}f(x):=\Big\langle{\bf u}_y,\frac{x f(x)-yf(y)}{x-y}\Big\rangle\,.
$$
(The notation ${\bf u}_y$ means that the functional ${\bf u}$ acts on polynomials of the variable $y$.)
Furthermore, if the polynomials $A$, $C$, and $D$ appearing
in (\ref{StieltSC}) are co-prime (i.e., there is no common zero to
these three polynomials), then the class of ${\bf u}$ is given by
\begin{equation}\label{def-classSC-Stieltjes}
s=\max\{\mbox{\rm deg}\,C-1,\mbox{\rm deg}\,D\}\;
\end{equation}
and the polynomials $\Phi$ and $\Psi$ that appear in (\ref{EqDistSC1}) are
\begin{equation}\label{phipsi}
\Phi(z)=q^{-{\small\mbox{deg}} A}\big(h_q A)(z)\; \mbox{and}\;
\Psi(z)=q^{-{\small\mbox{deg}} A}\left\{\big(H_q
A)(z)+q^{-1}C(z)\right\}.
\end{equation}

Table \ref{Table2} gives the polynomials $A$, $C$ and $D$ appearing in the
$q-$difference equation fulfilled by the formal Stieltjes series for the
$H_q-$classical functionals $(s=0)$ corresponding to the little
$q-$Laguerre OPS and to the little $q-$Jacobi OPS, given in Table \ref{Table1}.

{\begin{table}
\centering
\scriptsize
\begin{tabular}{|>{\columncolor[gray]{0.95}}c|c|c|c|}
\hline \rowcolor[gray]{0.95}
\rule{0pt}{1.2em} $p_n(x)$ & $A$ & $C$ & $D$ \\
\hline
\rule{0pt}{1.2em} $L_n(x;a|q)$ & $x$ & $qa^{-1}(q-1)^{-1}(x-1+aq)-q$& $u_0qa^{-1}(q-1)^{-1}$ \\[0.5em]
\rule{0pt}{1.2em} $U_n(x;a,b|q)$ & $x\big(x-b^{-1}\big)$ & $qa^{-1}b^{-1}(q-1)^{-1}\big(\big(abq^2-1\big)x+1-aq\big)$&$u_0\big(qa^{-1}b^{-1}(q-1)^{-1}\big(abq^2-1\big)-q^2\big)$\\ \rule{0pt}{1.2em} &&$\qquad-q^2\big(q^{-1}+1\big)x+qb^{-1}$ &  \\
\hline
\end{tabular}
\medskip
\caption{The polynomials $A$, $C$, and $D$ appearing in equation (\ref{StieltSC})
corresponding to the families in Table \ref{Table1}.}\label{Table2}
\end{table}}

\subsection{OP via polynomial mappings}\label{sec-poly-map}
Concerning the study of polynomial mappings in the framework of the theory of
OP, several works deal with the analysis of quadratic and cubic transformations (see e.g.
\cite{BarrucandDickinson,ChiharaBUMI,PacoGabriela,PacoZeLAA,PacoZePortMath,PacoZeCubic1,PacoZeCubic2,TounsiBouguerra,TounsiRaddaoui}).
For a general polynomial mapping, the corresponding sequences of OP have been studied by Geronimo and Van Assche \cite{JerVan},
Charris, Ismail, and Monsalve \cite{CharrisIsmail,CharrisIsmailMonsalve}, Peherstorfer \cite{Peherstorfer}, and
de Jesus and Petronilho \cite{MarcioPetronilhoJAT}. In order to describe this mapping,
let $\{p_n(x)\}_{n\geq0}$ be a monic OPS, characterized by its three-term recurrence relation,
 expressed in terms of blocks as
\begin{equation}\label{pnblock1}
\begin{array}r
(x-b_n^{(j)})p_{nk+j}(x)=p_{nk+j+1}(x)+a_n^{(j)}p_{nk+j-1}(x)
\qquad \qquad \qquad \\
\rule{0pt}{1.2em} (j=0,1,\dots, k-1 \, ; \quad n=0,1,2,\ldots)\, ,
\end{array}
\end{equation}
satisfying initial conditions $p_{-1}(x):=0$ and $p_0(x):=1$.
Without loss of generality, we take $a_0^{(0)}:=1$.
In general, the coefficients $a_n^{(j)}$'s and $b_n^{(j)}$'s are complex numbers
with $a_n^{(j)}\neq0$ for every $n$ and $j$.
As a consequence, we can construct determinants $\Delta_n(i,j;x)$,
as introduced by Charris and Ismail in \cite{CharrisIsmail},
and by  Charris, Ismail, and Monsalve in \cite{CharrisIsmailMonsalve}, so that
\begin{equation}\label{Delt0}
\Delta_n(i,j;x):=\left\{
\begin{array}{cl}
0 & \mbox{if $j<i-2,$} \\
\rule{0pt}{1.2em}
1  & \mbox{if $j=i-2,$} \\
\rule{0pt}{1.5em} x-b_n^{(i-1)}  & \mbox{if $j=i-1,$}
\end{array}
\right.
\end{equation}
and, if $j\geq i\geq 1$,
\begin{equation}\label{Delt1}
\Delta_n(i,j;x):=\left|
\begin{array}{cccccc}
x-b_n^{(i-1)} & 1 & 0 &  \dots & 0 & 0  \\
a_n^{(i)} & x-b_n^{(i)} & 1 &  \dots & 0 & 0 \\
0 & a_n^{(i+1)} & x-b_n^{(i+1)} &   \dots & 0 & 0 \\
\vdots & \vdots & \vdots  & \ddots & \vdots & \vdots \\
0 & 0 & 0 &  \ldots & x-b_n^{(j-1)} & 1 \\
0 & 0 & 0 &  \ldots & a_n^{(j)} & x-b_n^{(j)}
\end{array}
\right| \, ,
\end{equation}
for every $n\in\mathbb{N}_0$.
These determinants play a key role in the theory of OP via polynomial mappings.
Taking into account that $\Delta_n(i,j; x)$ is a polynomial
whose degree may exceed $k$, and since in (\ref{pnblock1})
the coefficients $a_n^{(j)}$'s and $b_n^{(j)}$'s were
defined only for $0\leq j\leq k-1$, we adopt the convention
\begin{equation}
b_n^{(k+j)}:=b_{n+1}^{(j)}\; ,\quad a_n^{(k+j)}:=a_{n+1}^{(j)}
\quad (i,j,n\in\mathbb{N}_0)\;,
\label{convention1}
\end{equation}
and so the following useful equality holds
\begin{equation}
\Delta_n(k+i,k+j;x)=\Delta_{n+1}(i,j;x)\;.
\label{convention2}
\end{equation}

\begin{theorem}\label{teobk1p2}\cite[Theorem 2.1]{MarcioPetronilhoJAT}
Let $\{p_n(x)\}_{n\geq0}$ be a monic OPS characterized by the general
blocks of recurrence relations $(\ref{pnblock1})$. Fix
$r_0\in\mathbb{C}$, $m\in\mathbb{N}_0$, $k\in\mathbb{N}$, and
$k\geq2$, with $0\leq m\leq k-1$. Then, there exist polynomials
$\pi_k(x)$ and $\theta_m(x)$ of degrees $k$ and $m$, respectively, and a
monic OPS $\{q_n(x)\}_{n\geq0}$ such that $q_1(0)=-r_0$ and
\begin{equation}
p_{kn+m}(x)=\theta_m(x)\, q_n(\pi_k(x)) \;,\quad n=0,1,2,\ldots
\label{pnblock5p2}
\end{equation}
if and only if the following  conditions hold:
\begin{enumerate}
\item[{\rm (i)}]
$b_n^{(m)}$ is independent of $n$ for $n\geq0$;
\item[{\rm (ii)}]
$\Delta_n(m+2,m+k-1;x)$ is independent of $n$ for $n\geq0$ and for every $x$;
\item[{\rm (iii)}]
$\theta_m (x)\equiv p_m (x)$ and $\Delta_0(m+2,m+k-1;x)$ is divisible by $\theta_m(x)$, i.e., there
exists a polynomial $\eta_{k-1-m}(x)$ of degree $k-1-m$ such that
$$
\Delta_0(m+2,m+k-1;x)=\theta_m(x)\,\eta_{k-1-m}(x)\, ;
$$
\item[{\rm (iv)}]
$r_n(x)$ is independent of $x$ for every $n\geq1$, where
$$
\begin{array}l
\quad\quad r_n(x):=
a_{n}^{(m+1)}\Delta_{n}(m+3,m+k-1;x)-a_0^{(m+1)}\Delta_{0}(m+3,m+k-1;x) \\
\rule{0pt}{1.2em} \qquad\qquad\qquad
+a_{n}^{(m)}\Delta_{n-1}(m+2,m+k-2;x)-a_0^{(m)}\Delta_{0}(1,m-2;x)\,\eta_{k-1-m}(x)\;.
\end{array}
$$
\end{enumerate}
Under such conditions, the polynomials $\theta_m(x)$ and $\pi_k(x)$ are explicitly given by
\begin{equation}\label{Pimka}
\begin{array}{l}
\pi_k(x)=\Delta_0(1,m;x)\,\eta_{k-1-m}(x)-a_0^{(m+1)}\,\Delta_0(m+3,m+k-1;x)+r_0
\; , \\ [0.5em]
\theta_m(x)=\Delta_0(1,m-1;x)\equiv p_m(x)\;,
\end{array}
\end{equation}
and the monic OPS $\{q_n(x)\}_{n\geq0}$ is generated by the three-recurrence relation
\begin{equation}
q_{n+1}(x)=\left(x-r_n\right)q_{n}(x)-s_n q_{n-1}(x) \, , \quad
n=0,1,2,\ldots  \label{pnblock4p2}
\end{equation}
with initial conditions $\, q_{-1}(x)=0\,$ and $\, q_0(x)=1 \,$,
where
\begin{equation}\label{rnsn1}
r_n:=r_0+r_n(0)\; ,\quad s_n:=a_n^{(m)}a_{n-1}^{(m+1)}\cdots a_{n-1}^{(m+k-1)}\; ,
\quad n=1,2,\ldots \, .
\end{equation}
Moreover, for each $j=0,1,2,\ldots,k-1$ and all $n=0,1,2,\ldots$,
\begin{equation}
\begin{array}l
\displaystyle p_{kn+m+j+1}(x)=\frac{1}{\eta_{k-1-m}(x)}\,\left\{ \rule{0pt}{1.2em}
 \Delta_n(m+2,m+j;x)\, q_{n+1}(\pi_k(x)) \right. \qquad \\
\rule{0pt}{1.5em} \qquad\hfill \left. + \left(\prod_{i=1}^{j+1} a_n^{(m+i)}\right)
 \Delta_n(m+j+3,m+k-1;x)\, q_{n}(\pi_k(x))\,\right\} \, .
\end{array}
\label{pnblockmp2}
\end{equation}
\end{theorem}

\begin{lemma}\cite[Lemma 3.3]{KMP}\label{Stieltjes-seriesA}
Under the conditions of Theorem $\ref{teobk1p2}$, the formal
Stieltjes series $S_{\bf u}(z):=-\sum_{n=0}^\infty u_n/z^{n+1}$
and $S_{\bf v}(z):=-\sum_{n=0}^\infty v_n/z^{n+1}$ associated with
the regular moment linear functionals ${\bf u}$ and ${\bf v}$ with
respect to which $\{p_n(x)\}_{n\geq0}$ and $\{q_n(x)\}_{n\geq0}$ are
orthogonal (resp.) are related by
\begin{equation}\label{SuSv}
S_{\rm u}(z)=\frac{u_0}{v_0}\,\frac{-v_0\Delta_0(2,m-1;z)+ \left(
\prod_{j=1}^m a_0^{(j)} \right)\,\eta_{k-1-m}(z)\,S_{\bf
v}(\pi_k(z))}{\theta_m(z)}\, .
\end{equation}
\end{lemma}

\section{Polynomial mappings and $H_q-$semiclassical OP}\label{poly-semi-OPS}

For fixed $\pi\in\mathcal{P}$, let
$\sigma_{\pi}:\mathcal{P}\to\mathcal{P}$ be the linear operator
such that $\sigma_\pi[f]:=f\circ\pi$ for every $f\in\mathcal{P}$,
and define $\sigma_{\pi}^*:\mathcal{P}^*\to\mathcal{P}^*$ by
duality. Henceforth,
$$
\sigma_{\pi}[f](x):=f\big(\pi(x)\big)\;,\quad
\langle\sigma_{\pi}^*(\textbf{u}),f\rangle:=\langle\textbf{u},\sigma_{\pi}[f]\rangle\;,
\quad f\in\mathcal{P}\;,\; \textbf{u}\in\mathcal{P}^\prime\,.
$$

\begin{lemma}\label{lemma1}
For fixed $f\in\mathcal{P}$ and $\emph{\textbf{u}}\in\mathcal{P}^\prime$,
the following relations hold:
\begin{eqnarray}
\label{re1} f\,\sigma_{x^k}^*(\emph{\textbf{u}})=\sigma_{x^k}^*\big(\sigma_{x^k}[f]\emph{\textbf{u}})\,; \\ [0.25em]
\label{re2} H_q\left(\,\sigma_{x^k}[f]\right)(x)=[k]_qx^{k-1}\sigma_{x^k}\big[H_{q^k}f](x)\,; \\ [0.25em]
\label{re3} \sigma_{x^k}^*(H_q\emph{\textbf{u}})=[k]_q\,H_{q^k}\,\big(\sigma_{x^k}^*\big(x^{k-1}\emph{\textbf{u}}))\,.
\end{eqnarray}
\end{lemma}
\begin{proof}
Relation (\ref{re1}) was stated in \cite[Lemma 2.1]{TounsiRaddaoui} for $k=3$,
being the proof similar for any positive integer number $k$.
Relation (\ref{re2}) holds, since
$$
H_q\big(\sigma_{x^k}[f]\big)(x)=\dfrac{f\big(q^kx^k\big)-f\big(x^k\big)}{(q-1)x}=\big(H_{q^k}
f\big)\big(x^k\big)[k]_qx^{k-1}=[k]_qx^{k-1}\sigma_{x^k}\big[H_{q^k}f\big](x)\,.
$$
Finally, (\ref{re3}) follows from (\ref{re2}) taking into account the equality
$$\left\langle\sigma_{x^k}^*\big(H_q\textbf{u}\big),f\right\rangle=
-\left\langle\textbf{u},\big(H_{q}\big(\sigma_{x^k}[f]\big)\right\rangle\;.$$
\end{proof}
\begin{lemma}{\rm\cite[Lemma 3.2]{KMP}}\label{Plemma}
Let $\{p_n (x)\}_{n\geq0}$ and $\{q_n(x)\}_{n\geq0}$ be two monic OPS
satisfying
$$
p_{nk}(x)=q_n\left(x^k\right)\;,\quad n=0,1,2,\ldots\,.
$$
Let ${\bf u}$ and ${\bf v}$ be the regular functionals in $\mathcal{P}^*$
with respect to which $\{p_n(x)\}_{n\geq0}$ and $\{q_n(x)\}_{n\geq0}$ are orthogonal (resp.), and
let $\{{\bf a}_n\}_{n\geq0}$ and $\{{\bf b}_n\}_{n\geq0}$ be the associated dual basis.
Then the following relations hold
\begin{align}
\label{rel1} \sigma_{x^k}^*\left({\bf a}_{nk+j}\right)
=\delta_{j,0}{\bf b}_{n}
\quad (j=0,1,\ldots,k-1, \; n=0,1,2,\ldots)\;, \\
\label{Mrel1}
\sigma_{x^k}^*\left(p_j\emph{\textbf{u}}\right)=\delta_{j,0}\,
v_0^{-1} u_0\emph{\textbf{v}} \quad (j=0,1,\ldots,k-1)\;.
\end{align}
\end{lemma}

\begin{lemma}\label{Stieltjes-series}
Let $\{p_n (x)\}_{n\geq0}$ and $\{q_n(x)\}_{n\geq0}$ be two monic OPS satisfying
$$
p_{nk}(x)=q_n\left(x^k\right)\;,\quad n=0,1,2,\ldots\;.
$$
(Hence one has $\pi_k(x):=x^k$, $\theta_m\equiv1$, and $m=0$ in Theorem $\ref{teobk1p2}$,
with conditions {\rm (i)}--{\rm (iv)} therein.)
Let ${\bf u}$ and ${\bf v}$ be the regular functionals in $\mathcal{P}^*$
with respect to which $\{p_n(x)\}_{n\geq0}$ and $\{q_n(x)\}_{n\geq0}$ are orthogonal (resp.).
Then, the associated formal Stieltjes series $S_{\bf u}(z)$
and $S_{\bf v}(z)$ satisfy
\begin{equation}\label{SuSvq}
[k]_{q^{-1}}z^{k-1}\eta_{k-1}(q^{-1}z)\big(H_{q^{-k}}S_{\bf v}\big)(z^k)
=\frac{v_0}{u_0}\,\big(H_{q^{-1}}\,S_{\bf u}\big)(z)-\big(H_{q^{-1}}\eta_{k-1}\big)(z)\,S_{\bf v}(z^k)\, .
\end{equation}
\end{lemma}

\begin{proof} By Lemma \ref{Stieltjes-seriesA} we have
\begin{equation}\label{Sv}
S_{\textbf{u}}(z)=\frac{u_0}{v_0}\eta_{k-1}(z)S_{\textbf{v}}(z^k)\,,
\end{equation}
hence
$$
\begin{array}{l}
\big(H_{q^{-1}}S_{\textbf{u}}\big)(z)
=\dfrac{S_{\textbf{u}}\big(q^{-1}z\big)-S_{\textbf{u}}(z)}{\big(q^{-1}-1\big)z}
=\dfrac{u_0}{v_0}\dfrac{\eta_{k-1}\big(q^{-1}z\big)S_{\textbf{v}}\big(q^{-k}z^k\big)-
\eta_{k-1}(z)S_{\textbf{v}}(z^k)}{\big(q^{-1}-1\big)z}\\ [1.2em]
\quad=\dfrac{u_0}{v_0}\eta_{k-1}\big(q^{-1}z\big)\dfrac{S_{\textbf{v}}\big(q^{-k}z^k\big)-
S_{\textbf{v}}(z^k)}{\big(q^{-k}-1\big)z^k}[k]_{q^{-1}}z^{k-1}+\dfrac{u_0}{v_0}S_{\textbf{v}}(z^k)
\dfrac{\eta_{k-1}\big(q^{-1}z\big)-\eta_{k-1}(z)}{\big(q^{-1}-1\big)z}\\ [1.2em]
\quad=\dfrac{u_0}{v_0}\left(\eta_{k-1}\big(q^{-1}z\big)\big(H_{q^{-k}}S_{\textbf{v}}\big)\big(z^k\big)
[k]_{q^{-1}}z^{k-1}+S_{\textbf{v}}(z^k)\big(H_{q^{-1}}\eta_{k-1}\big)(z)\right)\;.
\end{array}
$$
\end{proof}

\begin{lemma}\cite[Lemma 3.4]{KMP}\label{baseBpoly}
Let $\phi(x)$ be a polynomial and
$\mathcal{B}_k:=\{p_0,p_1,\ldots,p_{k-1}\}$ a finite simple set of
polynomials (i.e., $\deg p_j=j$ for all $j$). Then, to the pair
$\big(\phi,\mathcal{B}_k\big)$ we may associate $k$ polynomials
$\phi_0(x),\phi_1(x),\ldots,\phi_{k-1}(x)$, with %$\varphi_j$ not necessarily of degree $j$,
$\deg\phi_j(x)\leq\lfloor(\deg\phi(x))/k\rfloor$ for all $j=0,1,\ldots,k-1$, such that
\begin{equation}\label{phi-pi-u}
\phi(x)=\sum_{j=0}^{k-1}p_j(x) \sigma_{x^k}[\phi_j]\,.
\end{equation}
\end{lemma}

\begin{theorem}\label{T1}
Let $\{p_n (x)\}_{n\geq0}$ and $\{q_n(x)\}_{n\geq0}$ be monic OPS satisfying
\begin{equation}\label{TrPol}
p_{nk}(x)=q_n\left(x^k\right)\;,\quad n=0,1,2,\ldots\,.
\end{equation}
Then the following holds:
\begin{itemize}
\item[{\rm (i)}] If $\{p_n(x)\}_{n\geq0}$ is $H_q-$ semiclassical of
class $s$, then $\{q_n(x)\}_{n\geq0}$ is $H_{q^k}-$ semiclassical of
class $\widetilde{s}$, with
$\widetilde{s}\leq\left\lfloor\,s/k\,\right\rfloor$. \item[{\rm
(ii)}] If $\{q_n(x)\}_{n\geq0}$ is $H_{q^k}-$ semiclassical of class
$\widetilde{s}$, then $\{p_n(x)\}_{n\geq0}$ is $H_q-$ semiclassical
of class $s$, with $s\leq (\widetilde{s}+3)k-3$.
\end{itemize}
\end{theorem}

\begin{proof}
Denote by ${\bf u}$ and ${\bf v}$ the regular linear functionals with respect
to which $\{p_n(x)\}_{n\geq0}$ and $\{q_n(x)\}_{n\geq0}$ are OPS, respectively.

(i) Assume that $\{p_n(x)\}_{n\geq0}$ is $H_q-$ semiclassical of
class $s$.
Then there exist two nonzero polynomials $\Phi(x)$ and $\Psi(x)$, with ${\rm deg}\,\Psi(x)\geq 1$, such that
\begin{equation}\label{EqD}
H_q\left(\Phi\textbf{u}\right)=\Psi\textbf{u}\; ,
\end{equation}
being $s=\max\left\{{\rm deg}\,\Phi -2, {\rm deg}\,\Psi-1\right\}$.
Set $\ell:=1+\lfloor\,s/k\,\rfloor$ and $p:=\ell k-1-s$. Then $p\in\mathbb{N}_0$ and we deduce
\begin{equation}\label{EquD2A}
\sigma_{x^k}^*\left(x^p\,H_q\left(\Phi\textbf{u}\right)\right)
=\left\{\begin{array}{lcl}
[k]_q\,H_{q^k}\left(\sigma_{x^k}^*\left(x^{k-1}\Phi\,\textbf{u}\right)\right)\;,&{\rm
if}&p=0,\\\\
q^p[k]_q\,H_{q^k}\left(\sigma_{x^k}^*\left(x^{k+p-1}\Phi\,\textbf{u}\right)\right)-
[p]_q\sigma_{x^k}^*\left(x^{p-1}\Phi\,\textbf{u}\right)\;,\;&{\rm if}& p\geq 1\;.
\end{array}\right.
\end{equation}
In fact, assume that $p\geq1$.
Then for each $f\in\mathcal{P}$ we have
$$\begin{array}{l}
\left\langle\sigma_{x^k}^*\left(x^p\,H_q\left(\Phi\textbf{u}\right)\right),f\right\rangle=
\left\langle\,H_q\left(\Phi\textbf{u}\right),x^p\,f(x^k)\right\rangle\\[0.5em]
\quad =-\left\langle\,\Phi\textbf{u},H_q\left(x^p\,f(x^k)\right)\right\rangle
=-\left\langle\,\Phi\textbf{u},\dfrac{q^px^pf(q^kx^k)-x^pf(x^k)}{(q-1)x}\right\rangle\\[1.25em]
\quad=-q^p[k]_q\left\langle\,x^{k+p-1}\Phi\textbf{u},\big(H_{q^k}f\big)(x^k)\right\rangle-[p]_q
\left\langle\,x^{p-1}\Phi\textbf{u},f(x^k)\right\rangle\\[1.0em]
\quad=-q^p[k]_q\left\langle\sigma_{x^k}^*\left(x^{k+p-1}\,\Phi\textbf{u}\right),H_{q^k}f\right\rangle-
[p]_q\left\langle\sigma_{x^k}^*\left(x^{p-1}\,\Phi\textbf{u}\right),f\right\rangle\\[1.0em]
\quad=\left\langle
q^p[k]_qH_{q^k}\left(\sigma_{x^k}^*\left(x^{k+p-1}\,\Phi\textbf{u}\right)\right)-
[p]_q\sigma_{x^k}^*\left(x^{p-1}\,\Phi\textbf{u}\right),f\right\rangle\;,
\end{array}$$
which proves (\ref{EquD2A}) for $p\geq1$. The proof is similar for $p=0$.
It follows from (\ref{EqD}) and (\ref{EquD2A}) that
\begin{equation}\label{EquD2}H_{q^k}\left(\sigma_{x^k}^*\left(x^{k+p-1}\Phi\textbf{u}\right)\right)
=\left\{\begin{array}{lcl}
[k]_q^{-1}\sigma_{x^k}^*\left(\Psi\,\textbf{u}\right)\;,&{\rm if}&p=0\\[0.4em]
q^{-p}[k]_q^{-1}\sigma_{x^k}^*\left(\left(x^p\Psi+[p]_qx^{p-1}\Phi\right)\,\textbf{u}\right)
\;,\;&{\rm if}& p\geq 1\;.\end{array}\right.\end{equation}%\end{equation}
Next, Lemma \ref{baseBpoly} ensures the existence of polynomials $f_j(x)$ ($j=0,1,\ldots,k-1$),
with each $f_j(x)$ not necessarily of degree $j$, fulfilling
\begin{equation}\label{Soma1}
x^{k+p-1} \Phi (x)=\sum_{j=0}^{k-1}p_j(x)\sigma_{x^k}[f_j]\;.
\end{equation}
Applying the operator $\sigma_{x^k}^*$ and using Lemma
\ref{Plemma}, we obtain
\begin{equation}\label{EQ2}
\sigma_{x^k}^*\left(x^{k+p-1} \Phi\textbf{u}\right)
=\sum_{j=0}^{k-1}\sigma_{x^k}^*\big(p_j\sigma_{x^k}[f_j]\textbf{u}\big)
=v_0^{-1}u_0 f_0\textbf{v}\; .
\end{equation}
Similarly, consider polynomials $g_j(x)$ ($j=0,1,\ldots,k-1$),
with each $g_j(x)$ not necessarily of degree $j$, such that
\begin{equation}\label{Soma2}
\displaystyle
q^{-p}[k]_q^{-1}\big(x^p\,\Psi(x)+[p]_q\,x^{p-1}\,\Phi(x)\big)=\sum_{j=0}^{k-1}p_j (x)\sigma_{x^k}[g_j]\,,
\end{equation}
and proceed as above to deduce
\begin{equation}\label{EQ3}
\sigma_{x^k}^*\left(q^{-p}[k]_q^{-1}\big(x^p\,\Psi +[p]_q\,x^{p-1}\,\Phi\big)\textbf{u}\right)
=v_0^{-1}u_0 g_0\textbf{v}\;.
\end{equation}
From (\ref{EquD2}), (\ref{EQ2}), and (\ref{EQ3}), we obtain
\begin{equation}\label{EqD2}
H_{q^k}\left(f_0\textbf{v}\right)=g_0\textbf{v}\,.
\end{equation}
Since $s=\max\left\{{\rm deg}\,\Phi -2, {\rm deg}\,\Psi-1
\right\}$, then either ${\rm deg}\,\Phi=s+2$ and ${\rm deg}\,\Psi\leq s+1$ or else ${\rm
deg}\,\Phi<s+2$ and ${\rm deg}\,\Psi=s+1$. In the first case,
the polynomial appearing in the left-hand side of (\ref{Soma1})
has degree $(\ell+1)k$, hence from the right-hand side of
(\ref{Soma1}) we deduce ${\rm deg}\,f_0=\ell+1\geq2$. In the
second case, the polynomial appearing in the left-hand side
of (\ref{Soma2}) has degree $\ell k$, hence ${\rm
deg}\,g_0=\ell\geq1$. We conclude that, in any situation, at least
one of the polynomials $f_0$ or $g_0$ is different from zero.
Thus, since ${\bf v}$ is regular and fulfills (\ref{EqD2}), it
follows from Proposition \ref{PropositionPhiPsi} that $\textbf{v}$
is $H_{q^k}-$semiclassical (being both $f_0$ and $g_0$ different
from zero, and ${\rm deg}\,g_0\geq1$). It remains to prove that
the class $\widetilde{s}$ of $\textbf{v}$ satisfies
$\widetilde{s}\leq\left\lfloor\,s/k\,\right\rfloor$. Notice first
that
$$
k\deg f_0\leq\max_{0\leq j\leq k-1}\left\{j+k\deg f_j\right\}
=\mbox{\rm deg}\,\left\{x^{k+p-1}\,\Phi\,\right\} \leq(\ell+1)k
\;,
$$
where the equality follows from (\ref{Soma1}) and the last inequality
holds since $p=\ell k-1-s$ and $\deg\Phi\leq s+2$, hence $\deg
f_0\leq\ell+1$. In the same way, using (\ref{Soma2}), we deduce
$$
k\deg g_0 \leq
\deg\left\{x^p\,\Psi+[p]_q\,x^{p-1}\,\Phi\right\}\leq\ell k\;,
$$
so $\deg g_0\leq\ell$. But, taking into account the conclusions
of the discussion above involving the two possible
cases (concerning the degrees of $\Phi$ and $\Psi$), at least
one of the equalities $\deg f_0=\ell+1$ or $\deg g_0=\ell$ holds.
Therefore,
\begin{equation}\label{grau-fm-gm}
\max\left\{\deg f_0- 2, \deg g_0 -1\right\}
=\ell-1=\left\lfloor\,s/k\,\right\rfloor\,,
\end{equation}
and so from (\ref{EqD2}) we obtain
$\;\widetilde{s}\leq\max\left\{\deg f_0- 2, \deg g_0 -1\right\}
=\left\lfloor\,s/k\,\right\rfloor\,$.
\medskip

(ii) Assume now that $\{q_n(x)\}_{n\geq0}$ is $H_{q^k}-$semiclassical of
class $\widetilde{s}$. Then the associated formal Stieltjes
series, $S_{\bf v}(z):=-\sum_{n=0}^\infty v_n z^{-n-1}$, satisfies
the (formal) first order linear $q-$difference equation
\begin{equation}\label{FS1}
\widetilde{A}(z)\big(H_{q^{-k}} S_{\textbf{v}}\big)(z)
=\widetilde{C}(z)S_{\textbf{v}}(z)+\widetilde{D}(z)\,,
\end{equation}
with $\widetilde{A}$, $\widetilde{C}$, and $\widetilde{D}$
co-prime polynomials, $\widetilde{A}$ nonzero with $\deg\widetilde{A}\leq \widetilde{s}+2$, and
$\widetilde{s}=\max\{\deg\widetilde{C}-1, \deg\widetilde{D}\}$.
Replacing  in (\ref{FS1}) $z$ by $z^k$ and taking into account
(\ref{SuSvq}) and (\ref{Sv}), we deduce that $S_{\textbf{u}}(z)$ satisfies
$$A(z)\big(H_{q^{-1}}S_{\textbf{u}}\big)(z)=C(z)S_{\textbf{u}}(z)+D(z),$$
where \begin{equation}\label{ACD1}\begin{array}{l}
A(z):=v_0\eta_{k-1}(z)\widetilde{A}(z^k)\; ,\\
[0.4em]
C(z):=v_0\left([k]_{q^{-1}}z^{k-1}\eta_{k-1}(q^{-1}z)\widetilde{C}(z^k)
+\left(H_{q^{-1}}\eta_{k-1}\right)(z)\widetilde{A}(z^k)\right)\; ,\\ [0.4em]
D(z):=u_0[k]_{q^{-1}}z^{k-1}\eta_{k-1}(q^{-1}z)\eta_{k-1}(z)\widetilde{D}(z^k)\;.
\end{array}\end{equation}
Thus \textbf{u} is a $H_q-$semiclassical functional.
Let us prove that the class $s$ of \textbf{u} satisfies
$s\leq(\widetilde{s}+3)k-3$. Indeed, we have
$$
\begin{array}{l}
\deg
C\leq\max\{k-2+k\deg\widetilde{A},2(k-1)+k\deg\widetilde{C}\}\leq
k(\widetilde{s}+3)-2\; , \\ [0.25em] \deg
D=3(k-1)+k\deg\widetilde{D}\leq k(\widetilde{s}+3)-3\,.
\end{array}
$$
Henceforth, $s\leq\max\left\{\deg C- 1, \deg
D\right\}\leq(\widetilde{s}+3)k-3$.
\end{proof}

\begin{corollary}\label{cor-classical}
Let $\{p_n(x)\}_{n\geq0}$ and $\{q_n(x)\}_{n\geq0}$ be two monic OPS satisfying $(\ref{TrPol})$.
If $\{p_n (x)\}_{n\geq0}$ is $H_q-$semiclassical of class $s\leq k-1$,
then $\{q_n(x)\}_{n\geq0}$ is $H_{q^k}-$classical.
\end{corollary}

\begin{proof}
It is a straightforward consequence of part (i) in Theorem \ref{T1}.
\end{proof}

\begin{corollary}\label{cor-T2}
Let $\{p_n\}_{n\geq0}$ and $\{q_n\}_{n\geq0}$ be monic OPS satisfying $(\ref{TrPol})$, and let
${\bf u}$ and ${\bf v}$ be the corresponding regular functionals, respectively.
If there exist nonzero polynomials $\widetilde{\Phi}$ and $\widetilde{\Psi}$ such that
$D(\widetilde{\Phi}{\bf v})=\widetilde{\Psi}{\bf v}$, with
$\deg\widetilde{\Psi}\geq 1$, then
$$D(\Phi{\bf u})=\Psi{\bf u}\;,$$
where $\Phi$ and $\Psi$ are polynomials given by
$$\begin{array}l
\Phi(x):=q^{1-k}
\eta_{k-1}\big(qx\big)\widetilde{\Phi}\big(x^k\big)\; ,\quad\\ [0.25em]
\Psi(x):=q^{1-k}\left([k]_q
x^{k-1}\,\eta_{k-1}\big(q^{-1}x\big) \widetilde{\Psi}\big(x^k\big)
+\left(\big(H_q\eta_{k-1}\big)(x)+q^{-1}\big(H_{q^{-1}}\eta_{k-1}\big)(x)\right)
\widetilde{\Phi}\big(x^k\big)\right)\,.
\end{array}$$
\end{corollary}

\begin{proof} The statement follows immediately from
(\ref{phipsi}) and (\ref{ACD1}) in the proof of Theorem \ref{T1},
and taking into account the following relations:
$$
\begin{array}l
[k]_{q^{-1}}=[k]_qq^{1-k}\;,\\ [0.25em]
v_0^{-1}\big(H_qA\big)(x)=\eta_{k-1}(qx)\big(H_{q^k}\widetilde{A}\big)(x^k)[k]_qx^{k-1}+\widetilde{A}
\big(x^k\big)\big(H_{q}\eta_{k-1}\big)(x)\;,\\ [0.25em]
(qv_0)^{-1}C(x)=\eta_{k-1}(q^{-1}x)q^{-k}\widetilde{C}(x^k)[k]_qx^{k-1}+q^{-1}\widetilde{A}
\big(x^k\big)\big(H_{q^{-1}}\eta_{k-1}\big)(x)\;.\\
\end{array}
$$

\end{proof}

\section{Examples}\label{poly-cubic}

In this section we give examples of $H_q-$semiclassical OPS of classes $1$ and $2$ obtained via cubic transformations.
In particular, we confirm the results contained in the recent work \cite{TounsiRaddaoui}
where the authors considered the problem of determining all the $H_q-$semiclassical monic OPS of
class $1$, $\{p_n(x)\}_{n\geq0}$, such that the cubic decomposition
\begin{equation}\label{Prob-cubic}
p_{3n}(x)=q_n(x^3)\;, \quad n=0,1,2,\ldots
\end{equation}
holds, being $\{q_n(x)\}_{n\geq0}$ a monic OPS.
In \cite[Theorem 4.2]{TounsiRaddaoui} it was stated that property (\ref{Prob-cubic}) is fulfilled only if
$\{q_n(x)\}_{n\geq0}$ coincides with some specific family of
$H_{q^3}-$classical OPS (which have been determined explicitly, all the
possible families being special cases of Little $q^3-$Laguerre
and Little $q^3-$Jacobi polynomials, up to affine
changes of the variable). It is clear from Corollary
\ref{cor-classical} that, indeed, only $H_{q^3}-$classical OPS
$\{q_n(x)\}_{n\geq0}$ may appear as solutions of such a problem.
Moreover, we see immediately that considering the analog
problem demanding $\{p_n(x)\}_{n\geq0}$ to be $H_q-$semiclassical of
class $2$, then again only $H_{q^3}-$classical OPS
$\{q_n(x)\}_{n\geq0}$ may appear fulfilling such cubic
transformation. Thus, in the next we present several examples
involving as $\{q_n(x)\}_{n\geq0}$ the Little $q^3-$Laguerre or the Little $q^3-$Jacobi polynomials
and, in each case, we give the corresponding families $\{p_n(x)\}_{n\geq0}$
which are semiclassical of class $1$. Thus we recover the families presented in \cite{TounsiRaddaoui})
or of class $2$ by giving new examples.

\subsection{Description of the semiclassical families $\{p_n\}_{n\geq0}$ of class $s\leq2$}

We start by making the assumption that $\{p_n(x)\}_{n\geq0}$ is $H_q-$semiclassical of class $s\leq2$,
and noticing that (\ref{Prob-cubic}) corresponds to a polynomial mapping such that $k=3$ and $m=0$,
being $\pi_3(x)=x^3$ and $\theta_0\equiv1$.
Thus, by Corollary \ref{cor-classical}, $\{q_n(x)\}_{n\geq0}$ is a $H_{q^3}-$classical monic OPS.
We assume that $\{q_n(x)\}_{n\geq0}$ is (up to an affine change of variables)
one of the  families of the Little $q^3-$Laguerre or Little $q^3-$Jacobi polynomials,
described in Table \ref{Table1}.
We analyze these two cases separately.
Before performing this analysis, notice that
according to the expression of $\eta_{k-m-1}\equiv\eta_2$ given in Theorem \ref{teobk1p2}, one has
\begin{equation}
\label{eta2}
\eta_2(x)=\Delta_0(2,2)=x^2-\big(b_0^{(1)}+b_0^{(2)}\big)x+b_0^{(1)}b_0^{(2)}-a_0^{(2)}\;.
\end{equation}
On the other hand, by (\ref{Pimka}),
\begin{equation}
\label{pi3new}
x^3=\pi_3(x)=\big(x-b_0^{(0)}\big)\,\eta_{2}(x)-a_0^{(1)}\,\big(x-b_0^{(2)}\big)+r_0\;.
\end{equation}
Therefore, setting
\begin{equation}
\label{ktauAA}
\tau:=b_0^{(0)}\; ,\quad k_\tau:=a_0^{(1)}+\tau^2\; ,
\end{equation}
using (\ref{eta2}) and (\ref{pi3new}), we may write
\begin{equation}\label{polinomios}
\eta_2(x)=x^2+\tau x+k_\tau=p_2^{(1)}(x)\; .
\end{equation}
From now on we assume that $\{q_n(x)\}_{n\geq0}$ is $H_{q^3}-$classical and coincides with
one of the  families of the Little $q^3-$Laguerre or Little $q^3-$Jacobi polynomials.
Since $\{p_n(x)\}_{n\geq0}$ fulfils the cubic decomposition
(\ref{Prob-cubic}), it follows from the proof of Theorem \ref{T1}
that the formal Stieltjes series $S_{\bf u}(z)$ satisfies
$$A(z)\big(H_{q^{-1}}S_{\textbf{u}}\big)(z)=C(z)S_{\textbf{u}}(z)+D(z),$$
with \begin{equation}\label{ACD}\begin{array}{l}
A(z):=v_0\eta_{2}(z)\widetilde{A}(z^3)\; ,\\
[0.4em]
C(z):=v_0\left([3]_{q^{-1}}z^{2}\eta_{2}(q^{-1}z)\widetilde{C}(z^3)+\widetilde{A}(z^3)
\left(H_{q^{-1}}\eta_{2}\right)(z)\right)\; ,\\ [0.4em]
D(z):=u_0[3]_{q^{-1}}z^{2}\eta_{2}(q^{-1}z)\eta_{2}(z)\widetilde{D}(z^3)\;,
\end{array}\end{equation}
where the polynomials $\widetilde{A}(x)$, $\widetilde{C}(x)$, and $\widetilde{D}(x)$
are the polynomials $A$, $C$, and $D$ appearing in Table 2 with $q$ replaced by $q^3$.

Tables \ref{Table3} and \ref{Table4} describe all the possible monic OPS $\{p_n(x)\}_{n\geq0}$
semiclassical of classes $1$ or $2$ such that (\ref{Prob-cubic}) holds,
by giving the polynomials $\Phi$ and $\Psi$ appearing in the $q-$difference equation for
each associated functional ${\bf u}$, assuming that $\{q_n(x)\}_{n\geq0}$
is $H_{q^3}-$classical and coincides with one of the  families of the Little $q^3-$Laguerre
or Little $q^3-$Jacobi polynomials.
As we see on these tables, there are 13 possible cases.
The three cases corresponding to class $s=1$ are the ones that have been obtained
in \cite[Theorem 4.2]{TounsiRaddaoui}, up to affine change of the variables. For class $s=2$
the examples given in cases (4)--(13) are new.
Next we only give the details for obtaining the results in cases (1) and (4).
The procedure for the remaining cases is similar.

Let $\{q_n(x)\}_{n\geq0}$ be the Little $q^3-$Laguerre monic OPS, i.e.,
$q_n(x)\equiv L_n(x;a|q^3)$, for an arbitrary parameter $a$.
Write $\eta_2(z)=(z-z_1)(z-z_2)$.
Assume that $z_1\neq z_2$.
By (\ref{ACD}), $v_0z^2$ is a common factor of $A$, $C$, and $D$.
The division of these three polynomials by this factor gives
(for simplicity, we still use $A$, $C$, and $D$,
although in fact these are the polynomials obtained by dividing the above ones by the common factor)
\begin{equation}\label{FCDExemplo2}\begin{array}{l}
A(z):=z(z-z_1)(z-z_2)\; ,\\ [0.4em]
C(z):=[3]_{q^{-1}}\big(q^{-1}z-z_1\big)\big(q^{-1}z-z_2\big)\big(\ell\big(z^3-1+aq^3\big)-q^3\big)+\\ [0.4em]
\qquad\qquad \qquad\qquad \qquad +z\big(\big(q^{-1}+1\big)z-z_1-z_2\big)
\; ,\\ [0.4em]
D(z):=u_0[3]_{q^{-1}}\ell\big(q^{-1}z-z_1\big)\big(q^{-1}z-z_2\big)(z-z_1)(z-z_2)\;,
\end{array}
\end{equation}
where $\ell:=a^{-1}\big(q^3-1\big)^{-1}q^3$.
Now, we see that $C(0)=D(0)=0$ if and only if $z_1z_2=0$.
If $z_1=0$ (the reasoning hereafter is similar if $z_2=0$) then from (\ref{polinomios})
we obtain $z_2=-\tau\neq0$ (since we are assuming $z_1\neq z_2$).
Under such conditions, $z$ is a common factor of the polynomials $A$, $C$, and $D$ given by (\ref{FCDExemplo2}),
hence the division of these polynomials by $z$ yields
\begin{equation}\label{FCDExemplo3}\begin{array}{l}
A(z):=z(z+\tau)\; ,\\ [0.4em]
C(z):= q^{-1}(q-1)^{-1}a^{-1}\big(z^4+\tau q z^3+\big(aq^2-1\big)z+\tau q\big(a q-1\big)\big)\; ,\\ [0.5em]
D(z):=u_0q^{-1}(q-1)^{-1}a^{-1} z(z+\tau q) (z+\tau)\;.
\end{array}
\end{equation}
Now, for these polynomials (\ref{FCDExemplo3}), we have $C(0)=D(0)=0$ if and only if $a=q^{-1}$.
Under such conditions $z$ is a common factor of the polynomials $A$, $C$, and $D$ given by (\ref{FCDExemplo3}),
and so dividing these polynomials by $z$ we obtain
\begin{equation}\label{FCDExemplo4}\begin{array}{l}
A(z):=z+\tau\; ,\\ [0.4em]
C(z):=  (q-1)^{-1}\big(z^3+\tau qz^2+q-1\big)\; ,\\ [0.5em]
D(z):= u_0(q-1)^{-1}(z+\tau q)(z+\tau)\;.
\end{array}
\end{equation}
Then $C(-\tau)=D(-\tau)=0$ if and only if $\tau^3=-1$. Under such conditions,
$z+\tau$ is a common factor of $A$, $C$, and $D$ given by (\ref{FCDExemplo4}).
Hence, dividing by $z+\tau,$ we obtain
\begin{equation}\label{FCDExemplo5}\begin{array}{l}
A(z):=1\; ,\\ [0.4em]
C(z):=(q-1)^{-1}\big(z^2-\tau(1-q)z+\tau^2(1-q)\big)\; ,\\ [0.5em]
D(z):=u_0(q-1)^{-1}(z+\tau q)\;.
\end{array}
\end{equation}
We may conclude that if $z_1=0$, $z_2=-\tau\neq0$, $a=q^{-1}$, and $\tau^3=-1$,
then {\textbf{u}} is semiclassical of class $s=1$. Moreover, under these conditions,
from (\ref{phipsi}) we obtain
$$\Phi(x)=1\; ,\quad \Psi(x)=q^{-1}\big((q-1)^{-1}x^2+\tau x-\tau^2\big)\,.$$
This gives case (1) appearing in Table \ref{Table3} and recovers the first
solution presented in \cite[Theorem 4.2]{TounsiRaddaoui}.
If $z_1=0$, $z_2=-\tau\neq0$, $a=q^{-1}$, and $\tau^3\neq -1$,
then it is clear from (\ref{FCDExemplo4}) that {\textbf{u}} is semiclassical of
class $s=2$, and from (\ref{phipsi}) we deduce
$$\Phi(x)=x+\tau q^{-1}\; ,\quad \Psi(x)=q^{-2}(q-1)^{-1}\big(x^3+\tau q x^2+ q^2-1\big)\,.$$
This gives case (4) appearing in Table \ref{Table3}.
We also note that if $z_1=z_2$, then using a similar reasoning we can show that
the class of {\textbf{u}} is greater than two.

%\begin{landscape}
{\scriptsize
\centering
\begin{table}
\hspace*{-3em}
\begin{tabular}{|>{\columncolor[gray]{0.95}}c|c|c|c|c|c|}
\hline \rowcolor[gray]{0.95}
\rule{0pt}{1.2em} Case & $s$ & $q_n(x)$ &  $b_0^{(0)}$ & $a_0^{(1)}$ & Constraints  \\ [0.2em]
\hline
\rule{0pt}{1.2em} (1) & $1$ & $L_n(x;a|q^3)$ &  $\tau$ & $-\tau^2$ & $a=q^{-1}$ , $\tau^3=-1$  \\
%\hline
\rule{0pt}{1.2em} (2) && $U_n(x;a,b|q^3)$ & $\tau$ & $-\tau^2$ &   $a=q^{-1}$ , $\tau^3=-1$  \\
%\hline
\rule{0pt}{1.2em} (3) && $U_n(x;a,b|q^3)$ &  $\tau$ & $-\tau^2$ &
$a=q^{-1}$ , $\tau^3\neq-1$, $b=-\tau^{-3}q^{-3}$, $\tau\neq 0$   \\ [0.2em]
\hline
\rule{0pt}{1.2em} (4) & $2$ & $L_n(x;a|q^3)$  &  $\tau$ & $-\tau^2$ &  $a=q^{-1}\;,\;\tau\neq0\;,\;\tau^3\neq-1$  \\
%\hline
\rule{0pt}{1.2em} (5) && $L_n(x;a|q^3)$ &  $\tau$ & $-\tau^2$ &  $a\neq q^{-1}\;,\;\tau^3=-1$ \\
%\hline
\rule{0pt}{1.2em} (6) && $L_n(x;a|q^3)$ &  $\tau$ & $-\tau^2[3]_q/(1+q)^2$ &  $\tau^3=-(1+q)^3$ \\
%\hline
\rule{0pt}{1.2em} (7) && $U_n(x;a,b|q^3)$ & $\tau$ & $-\tau^2$ &  $a\neq q^{-1}$ , $\tau^3=-1$ \\
%\hline
\rule{0pt}{1.2em} (8) && $U_n(x;a,b|q^3)$ & $\tau$ & $-\tau^2$ &
$a=q^{-1}$ , $\tau^3\neq-1$, $b\neq-\tau^{-3}q^{-3}$, $\tau\neq 0$  \\
%\hline
\rule{0pt}{1.2em} (9) && $U_n(x;a,b|q^3)$ & $\tau$ & $-\tau^2$ &
$a\neq q^{-1}$ , $\tau^3\neq-1$, $b=-\tau^{-3}q^{-3}$, $\tau\neq 0$ \\
%\hline
\rule{0pt}{1.2em} (10) && $U_n(x;a,b|q^3)$ & $\tau$ & $-\tau^2[3]_q/(1+q)^2$ &  $\tau^3=-(1+q)^3$ \\
%\hline
\rule{0pt}{1.2em} (11) && $U_n(x;a,b|q^3)$ & $\tau$ & $-\tau^2[3]_q/(1+q)^2$ &  $\tau^3=-q^{-3}(1+q)^3$, $b=1$ \\
%\hline
\rule{0pt}{1.2em} (12) && $U_n(x;a,b|q^3)$ & $\tau$ & $-\tau^2[3]_q/(1+q)^2$ &
$\tau\neq0$, $\tau^3\neq-(1+q)^3$, $b=-(1+q)^3\tau^{-3}q^{-6}$ \\
[0.2em]
%\hline
\rule{0pt}{1.2em} (13) && $U_n(x;a,b|q^3)$ & $\tau$ & $-\big(c^2+\tau c+\tau^2\big)$ &  $c\neq0$, $c\neq -\tau q/(1+q)$, $c^2+\tau c+\tau^2\neq 0$,\\
%\hline
\rule{0pt}{1.1em}
 &&  &  &  & $(\tau+c)^3=-1$, $b=c^{-3}q^{-3}$ \\
%\hline
\hline
\end{tabular}\medskip
\caption{\small Description of all possible $H_q-$semiclassical OPS
$\{p_n(x)\}_{n\geq0}$ of class $s\leq2$ obtained via a cubic transformation
such that $p_{3n}(x)=q_n(x^3)$ for all $n\geq0$, being $\{q_n(x)\}_{n\geq0}$ either the
sequence of little $q^3-$Laguerre polynomials $\{L_n(x;a|q^3)\}_{n\geq0}$, or the sequence of little $q^3-$Jacobi
polynomials $\{U_n(x;a,b|q^3)\}_{n\geq0}$. This polynomial mapping depends on the choice
of the parameters $a$, $b$, $c$ and $\tau$, which may be chosen
arbitrarily in $\mathbb{C}$ subject to the given constraints.}\label{Table3}
\end{table}}
%\end{landscape}

{\begin{landscape}
\scriptsize
\centering
\begin{table}
\hspace*{-3em}
\begin{tabular}{|>{\columncolor[gray]{0.95}}c|c|c|}
\hline \rowcolor[gray]{0.95}
\rule{0pt}{1.2em} Case & $\Phi$ & $\Psi$ \\ [0.2em]
\hline
\rule{0pt}{1.2em} (1) &  $1$ & $q^{-1}\big((q-1)^{-1}x^2+\tau x-\tau^2\big)$ \\ [0.5em]
%\hline
\rule{0pt}{1.2em} (2) & $x^3-q^{-3}b^{-1}$ & $q^{-4}b^{-1}\big((q-1)^{-1}\big(q^4b-1\big)x^2-\tau x+\tau^2\big)$ \\ [0.5em]
%\hline
\rule{0pt}{1.2em} (3) & $x^3+q^{-1}(1-q)\tau x^2-q^{-1}(1-q)\tau^2x+q^{-1}\tau^3 $& $(q-1)^{-1}\big(1-q^{-4}b^{-1}\big)x^2-q^{-1}\tau x+q^{-1}\tau^2$ \\ [0.5em]
\hline
\rule{0pt}{1.2em} (4) & $x+\tau q^{-1}$ & $q^{-2}(q-1)^{-1}\big(x^3+\tau q x^2+q^2-1\big)$ \\ [0.5em]
%\hline
\rule{0pt}{1.2em} (5) & $x$ & $q^{-3}a^{-1}\big((q-1)^{-1}x^3+\tau x^2-\tau^2 x+q(q-1)^{-1}(aq^2-1)\big)$ \\ [0.5em]
%\hline
\rule{0pt}{1.2em} (6) & $x$ & $q^{-3}a^{-1}\big((q-1)^{-1}x^3+\tau x^2-\tau^2(1+q)^{-1}x+q^2(q-1)^{-1}(aq-1)\big)$ \\ [0.5em]
%\hline
\rule{0pt}{1.2em} (7) & $x^4-q^{-3}b^{-1}x$ & $q^{-6}(q-1)^{-1}b^{-1}a^{-1}
\big(\big(ab q^6-1\big)x^3+\tau(1-q)x^2+\tau^2(q-1)x+q-aq^3\big)$ \\ [0.5em]
%\hline
\rule{0pt}{1.2em} (8) & $x^4+q^{-1}\tau x^3-q^{-3}b^{-1}x-q^{-4}b^{-1}\tau$ & $q^{-5}(q-1)^{-1}b^{-1}
\big(\big(q^5b-1\big)x^3+\tau q\big(b q^3-1\big)x^2+\big(1-q^2\big)\big)$ \\ [0.5em]
%\hline
\rule{0pt}{1.2em} (9) &
$x^4+q^{-1}\tau(1-q)x^3+\tau^2q^{-1}(q-1)x^2+\tau^3q^{-1}x$ &
$q^{-6}(q-1)^{-1}b^{-1}a^{-1}
\big(\big(ab q^6-1\big)x^3+abq^5\tau(1-q)x^2+abq^5\tau^2(q-1)x$\\
&&$+abq^4\tau^3(q-1)+1-aq\big)$  \\ [0.5em]
%\hline

\rule{0pt}{1.2em} (10) & $x^4-q^{-3}b^{-1}x$ & $a^{-1}b^{-1}q^{-6}\big((q-1)^{-1}(abq^6-1)x^3-\tau x^2+(q+1)^{-1}\tau^2 x+q^2(q-1)^{-1}(1-aq)\big)$ \\ [0.5em]
\rule{0pt}{1.2em} (11) & $x^4+\tau q^{-1}(q+1)^{-1}(1-q)x^3+\tau^2 q^{-1}(q+1)^{-2}(q-1)x^2$&$a^{-1}q^{-4}\big(q^{-2}(q-1)^{-1}\big(aq^6-1\big)x^3-\tau q^{-1}(q+1)^{-1}\big(aq^4+1\big)x^2+\tau^2(q+1)^{-2}\big(aq^3+1\big)x$\\
&$+\tau^3 q^{-1}(q+1)^{-3}x$ &$+(q+1)^{-3}(q-1)^{-1}q^{-1}\big(aq^3\tau^3(q-1)
+(q+1)^3(1-a)\big)\big)$
\\ [0.5em]
%\hline
\rule{0pt}{1.2em} (12) & $x^4+\tau q^{-1}(1-q)x^3+\tau^2(q+1)^{-1}(q-1)x^2+q\tau^3(q+1)^{-3}x$ & $a^{-1}(q-1)^{-1}(q+1)^{-3}\big(\tau^3+a\big(q^3+1\big)+3aq(q+1)\big)x^3-\tau q^{-1} x^2$\\
&&$+\tau^2(q+1)^{-1}x+(q-1)^{-1}a^{-1}(q+1)^{-3}\tau^3(aq-1)$ \\ [0.5em]
%\hline
\rule{0pt}{1.2em} (13) & $x^4+cq^{-1}(q-1)x^3+c^2q^{-1}(q-1)x^2-c^3q^{-1}x$ & $a^{-1}b^{-1}q^{-6}\big((q-1)^{-1}(abq^6-1)x^3+\big(c\big(abq^5-1\big)-\tau\big)x^2$\\
&&$+\big(c^2\big(abq^5+1\big)+\tau(\tau+2c)\big)x-(q-1)^{-1}\big(c^3abq^4(q-1)+q(a-1)\big)$ \\ [0.5em]
\hline
\end{tabular}\medskip
\caption{\small The polynomials $\Phi$ and $\Psi$ appearing in the canonical
distributional $q-$difference equation $H_q(\Phi{\bf u})=\Psi{\bf u}$ satisfied
by the functional ${\bf u}$ with respect to which $\{p_n(x)\}_{n\geq0}$ is an OPS,
in accordance with each case described in Table \ref{Table3}.}\label{Table4}
\end{table}
\end{landscape}}

\subsection{Discrete measure representation}
Next, we provide a discrete measure representation for the functional ${\bf u}$ with respect to which $\{p_n(x)\}_{n\geq0}$
is a monic OPS (given by Tables 3 and 4, when $0<q<1$). These representations may be obtained using the following proposition, which
generalizes \cite[Lemma 4.3]{TounsiRaddaoui}.

\begin{lemma}\label{lemadiscret}
Let  $\{p_n(x)\}_{n\geq0}$ and $\{q_n(x)\}_{n\geq0}$ be monic OPS satisfying (\ref{Prob-cubic}), and let \emph{$\textbf{u}$} and \emph{$\textbf{v}$} be the corresponding regular linear functionals in $\mathcal{P}^*$ (respectively). Assume further that \emph{$\textbf{v}$} has a discrete measure representation
\begin{equation}\label{discretev1}
\emph{\textbf{v}}=\sum_{\ell=0}^{+\infty}a_{\ell}\delta_{ \mu_{\ell}^3}\;,
\end{equation}
%where $\{a_\ell\}_{\ell\geq0}$ is a sequence of nonnegative real numbers and $\{\mu_\ell\}_{\ell\geq0}$
%a sequence of real numbers subject to the condition
being $\{a_\ell\}_{\ell\geq0}$ and $\{\mu_\ell\}_{\ell\geq0}$
sequences of complex numbers, with $\mu_\ell\neq0$ for each $\ell=0,1,2,\ldots$, such that
\begin{equation}\label{discretev2}
\left|\sum_{\ell=0}^{+\infty}a_{\ell} \mu_{\ell}^{n-2}\right|<+\infty\;,\quad\forall n\in N_\tau\;,
%\quad n=0,1,2,\ldots\;.
\end{equation}
where $N_\tau:=\mathbb{N}$ if $k_\tau=0$, and $N_\tau:=\mathbb{N}_0$ if $k_\tau\neq0$,
and $k_\tau$ is defined by (\ref{ktauAA}).
Then \emph{$\textbf{u}$} has the discrete measure representation
\begin{equation}\label{discreteu}
\emph{\textbf{u}}=\frac{u_0}{v_0}\,\sum_{\ell=0}^{+\infty}\frac{a_{\ell}}{3 \mu_{\ell}^2}
\sum_{p=0}^2j^p\eta_2\left(j^p \mu_{\ell}\right)\delta_{j^p \mu_{\ell}}\,,\end{equation}
where $j=e^{2\pi i/3}$ and $\eta_2$ is the polynomial given by (\ref{polinomios}).
\end{lemma}

\begin{proof}
By Lemma \ref{baseBpoly}, for each polynomial $f(x)$,  there are polynomials $f_0(x)$, $f_1(x),$ and $f_2(x)$ such that the decomposition
\begin{equation}\label{decompf}
f(x)=f_0(x^3)+p_1(x)f_1(x^3)+p_2(x)f_2(x^3)
\end{equation}
holds. Therefore, by Lemma \ref{Plemma} we have
\begin{equation}\label{uvAA}
\langle \textbf{u},f\rangle=v_0^{-1}u_0\langle \textbf{v}, f_0\rangle\;.
\end{equation}
%Taking into account (\ref{discretev1}) and (\ref{discretev2}), we deduce
%\begin{equation}\label{vf1}
%\langle\textbf{v},f_0(x)\rangle=\sum_{\ell=0}^{+\infty}\frac{a_{\ell}}{3\mu_{\ell}^2}\sum_{p=0}^2\langle j^p\eta_2\left(j^p\mu_{\ell}\right)\delta_{j^p\mu_{\ell}},f_0(x^3)\rangle\,.
%\end{equation}
Since $\eta_2(x)=x^2+\tau x+k_\tau$, $j^3=j^6=1$, $j^4=j$, and $1+j+j^2=0$, we compute
$\eta_2\left(\mu_{\ell}\right)+j\eta_2\left(j\mu_{\ell}\right)+j^2\eta_2\left(j^2\mu_{\ell}\right)=3\mu_\ell^2$,
hence
\begin{equation}\label{vf2AA}
\sum_{p=0}^2\big\langle j^p\eta_2\left(j^p\mu_{\ell}\right)\delta_{j^p\mu_{\ell}},f_0(x^3)\big\rangle
=3\mu_\ell^2f_0\left(\mu_{\ell}^3\right)\;.
\end{equation}
Similarly, using $p_1(x)=x-\tau$ and $p_2(x)=x^2-\big(b_0^{(0)}+b_0^{(1)}\big)x+b_0^{(0)}b_0^{(1)}-a_0^{(1)}$,
together with (\ref{ktauAA}), we show that
\begin{equation}\label{vf2}
\sum_{p=0}^2\big\langle j^p\eta_2\left(j^p\mu_{\ell}\right)\delta_{j^p\mu_{\ell}},p_1(x)f_1(x^3)\big\rangle=0
\end{equation}
and
\begin{equation}\label{vf3}
\sum_{p=0}^2\big\langle j^p\eta_2\left(j^p\mu_{\ell}\right)\delta_{j^p\mu_{\ell}},p_2(x)f_2(x^3)\big\rangle=0\;.
\end{equation}
Thus (\ref{discreteu}) follows from (\ref{decompf})--(\ref{vf3}),
taking into account (\ref{discretev1}) and (\ref{discretev2}).
\end{proof}

\begin{remark}
If $k_\tau:=a^{(1)}_0+\tau^2=0,$ then we recover \cite[Lemma 4.3]{TounsiRaddaoui}.
\end{remark}

Lemma \ref{lemadiscret} may be applied to give a discrete measure representation of
the functional $\textbf{u}$ in each case quoted in Tables 3 and 4.
We will present two illustrative examples.
The first one recovers the first discrete measure representation for the functional ${\bf u}$
given in \cite[Theorem 4.4]{TounsiRaddaoui},
which corresponds to the case described in line (1) appearing in Tables \ref{Table3} and \ref{Table4}.
Using the data in line (1) of Table \ref{Table3}, we have $\eta_2(x)=x(x+\tau)$ and $k_\tau=0$.
From the discrete representation of the functional $\mathcal{L}(a,q)$ given in \cite[Appendix 1]{TounsiRaddaoui}, it is clear that conditions (\ref{discretev1}) and (\ref{discretev2}) are fulfilled for $0<q<1$, where $$a_{\ell}:=\big(q^2,q^3\big)_{\infty}\dfrac{q^{2\ell}}{\big(q^3,q^3\big)_{\ell}}\;,\quad
\mu_{\ell}:=q^{\ell}\;,\quad \ell=0,1,2,\ldots\;.$$
Therefore, by (\ref{discreteu}), we deduce
$$\textbf{u}=\dfrac{u_0}{3}\big(q^2,q^3\big)_{\infty}\sum_{\ell=0}^{+\infty}\dfrac{q^{\ell}}{\big(q^3,q^3\big)_{\ell}}
\Big(\big(q^{\ell}+\tau\big)\delta_{q^{\ell}}
+\big(q^{\ell}+j^2\tau\big)\delta_{jq^{\ell}}+\big(q^{\ell}+j\tau\big)\delta_{j^2q^{\ell}}\Big)\,,$$
recovering (up to an affine change of variables) the first solution presented in \cite[Theorem 4.4]{TounsiRaddaoui}.
Discrete measure representation for functionals ${\bf u}$ corresponding to other lines
in Tables \ref{Table3} and \ref{Table4} may be obtained by a similar process. For instance, for the functional ${\bf u}$ of class $s=2$ described in line ($13$), we have $\eta_2(x)=x^2+\tau x-c(c+\tau)$ and $k_\tau=-c(c+\tau)\neq0$.
Thus, for the discrete representation of the functional $\mathcal{U}(a,b,q)$ given in \cite[Appendix 1]{TounsiRaddaoui}, it is clear that conditions (\ref{discretev1}) and (\ref{discretev2}) are fulfilled for $0<q<1$ and $0<a<q^{-1}$, where $$a_{\ell}:=\dfrac{\big(aq^3,q^3\big)_{\infty}}{\big(ac^{-3}q^3,q^3\big)_{\infty}}
\dfrac{\big(c^{-3},q^3\big)_{\ell}}{\big(q^3,q^3\big)_{\ell}}\big(aq^3\big)^{\ell}\;,\quad\mu_{\ell}:=q^{\ell}\;,\quad
\ell=0,1,2,\ldots\;.$$
Hence, using Lemma \ref{lemadiscret}, we deduce
%that in this case a discrete measure representation for $\textbf{u}$ is
$$\begin{array}{l}\textbf{u}=\dfrac{u_0}{3}\dfrac{\big(aq^3,q^3\big)_{\infty}}{\big(ac^{-3}q^3,q^3\big)_{\infty}}\displaystyle\sum_{\ell=0}^{+\infty}
\dfrac{\big(c^{-3},q^3\big)_{\ell}}{\big(q^3,q^3\big)_{\ell}}\big(aq^3\big)^{\ell}\Big(\big(q^{2\ell}+\tau q^{\ell}-c(c+\tau)\big)\delta_{q^{\ell}}\\
\qquad\qquad +\big(q^{2\ell}+j^2\tau q^{\ell}-jc(c+\tau)\big)\delta_{jq^{\ell}}+\big(q^{2\ell}+j\tau q^{\ell}-j^2c(c+\tau)\big)\delta_{j^2q^{\ell}}\Big)\,.
\end{array}$$

\subsection{Further remarks}
We conclude this work with some remarks concerning the analysis presented here.
In the previous section, we started from the knowledge of the monic OPS $\{q_n(x)\}_{n\geq0}$,
and we found the corresponding monic OPS $\{p_n(x)\}_{n\geq0}$ satisfying the cubic
transformation (\ref{Prob-cubic}), requiring $\{p_n(x)\}_{n\geq0}$ to be $H_q-$semiclassical of class at most $2$.
We point out that, conversely, starting from a given
$H_q-$semiclassical monic OPS $\{p_n(x)\}_{n\geq0}$ of class at most $2$-- e.g.,
as described by Tables \ref{Table3} and \ref{Table4} --, we can find the corresponding
monic OPS $\{q_n(x)\}_{n\geq0}$ fulfilling (\ref{Prob-cubic}),
provided we know \emph{a priori} that such a cubic transformation exists.
We will illustrate this procedure considering the monic OPS $\{p_n(x)\}_{n\geq0}$
described in case (13) appearing in Tables \ref{Table3} and \ref{Table4}
(and so, we already know that a cubic transformation exists).
In this case, from Table \ref{Table3}, we have $b_0^{(0)}=\tau$ and $a_0^{(1)}=-\big(c^2+\tau c +\tau^2\big)$.
Since the polynomials $\Phi$ and $\Psi$ satisfy the $q-$difference equation
$H_q(\Phi{\bf u})=\Psi{\bf u}$, then $\langle{\bf u},\Psi\rangle=0$.
On the other hand, we also have $p_1(x)=x-\tau$,
$p_2(x)=\Delta_0(1,1;x)=x^2-(b_0^{(1)}+\tau)x+\tau b_0^{(1)}+c^2+\tau c+\tau^2$,
and $p_3(x)=\pi_3(x)-r_0=x^3-r_0$.
Therefore, from the system of four equations $\langle{\bf u},\Psi\rangle=\langle{\bf u},p_j\rangle=0$ ($j=1,2,3$),
we deduce, after some computations, that
$$r_0=c^3\big(c^3-a q^3\big)^{-1}\big(1-a q^3\big)\,.$$
Furthermore, using (\ref{eta2}), (\ref{pi3new}), and (\ref{polinomios}), we find
\begin{equation}\label{systR}
\begin{array}{c}
b_0^{(1)}=-\tau-b_0^{(2)}\;,\quad
b_0^{(2)}=\tau+ \dfrac{\tau^3-r_0}{a_0^{(1)}}\;, \quad a_0^{(2)}=b_0^{(1)}b_0^{(2)}-a_0^{(1)}-\tau^2\;.
\end{array}
\end{equation}
Replacing in (\ref{systR}) the above expression for $r_0$ and taking into account the constraints
appearing in case (13) of Table 3, we obtain
$$
\begin{array}{c}
b_0^{(1)}=c+\dfrac{a q^3\big(c^3-1\big)}{\big(c^3-a q^3\big)\big(c^2+\tau c+\tau^2\big)}\;,\quad
b_0^{(2)}=c+\dfrac{c^3\big(1-c^3\big)}{\big(c^3-a q^3\big)\big(c^2+\tau c+\tau^2\big)}\;,\\ [0.5em]
a_0^{(2)}=\dfrac{c^3q^3 a\big(1-c^3\big)}{\big(c^3-a q^3\big)^2\big(c^2+\tau c+\tau^2\big)^2}\;.
\end{array}
$$
On the other hand, using again the data in case (13) of Table 4, we may write
$$\begin{array}{c}
x^2\Phi(x)=x^6+A_1x^5+A_2x^4+A_3x^3\;,\quad
\Psi(x)=B_1x^3+B_2x^2+B_3x+B_4\,,
\end{array}$$
where the coefficients $A_i$ and $B_i$ are
$$
\begin{array}{c} A_1:=c\big(1-q^{-1}\big)\,,\quad A_2:=c^2\big(1-q^{-1}\big)\,, \quad A_3:=-c^3q^{-1}\;,\\ [0.5em]
B_1:=a^{-1}q^{-3}(q-1)^{-1}\big(a q^3-c^3\big)\,,\quad B_2:=a^{-1}q^{-3}c\big(a q^2-\tau c^2-c^3\big)\,,\\ [0.5em]
B_3:=a^{-1}q^{-3}c^2\big(a q^2+\tau^2 c+2\tau c^2+c^3\big)\,,\quad B_4:=a^{-1}q^{-2}(q-1)^{-1}c^3\big(a q-1\big)\,.
\end{array}
$$
The above computations yield
$$
\begin{array}{rcl}
x^2\Phi(x) &=& p_0(x)\big(x^6+\big(A_1\big(\tau^2+a_0^{(1)}\big)+A_2\tau+A_3\big)x^3\big) \\ [0.25em]
&& \quad+p_1(x)\big(A_2+\big(\tau+b_0^{(1)}\big)A_1\big)x^3+p_2(x)A_1 x^3,\\ [0.5em]
\Psi(x) &=& p_0(x)\big(B_1 x^3+B_2\big(\tau^2+a_0^{(1)}\big)+B_3\tau+B_4\big) \\ [0.25em]
&& \quad+p_1(x)\big(B_3+\big(\tau+b_0^{(1)}\big)B_2\big)+p_2(x)B_2\,.
\end{array}
$$
Therefore, using (\ref{Soma1}), (\ref{Soma2}) and (\ref{EqD2}) with $k=3$ and $m=p=0$
(we notice that while proving (\ref{Soma1}) and (\ref{Soma2}) we considered $p\geq1$,
but by direct inspection we see that the formulas also hold for $p=0$),
we concluded that
$$
\begin{array}{rcl}
f_0\big(x^3\big)&=&x^6+\big(A_1\big(\tau^2+a_0^{(1)}\big)+A_2\tau+A_3\big)x^3=x^3(x^3-c^3)\;,\\ [0.5em]
g_0\big(x^3\big)&=&[3]_q^{-1}\big(B_1x^3+B_2\big(\tau^2+a_0^{(1)}\big)+B_3\tau+B_4\big) \\ [0.25em]
&=&q^{-3}a^{-1}\big(q^3-1\big)^{-1}\big(\big(aq^3-c^3\big)x^3+c^3\big(1-aq^3\big)\big)\,.
\end{array}
$$
Hence the functional ${\bf v}$ fulfils
$$H_{q^3}\big(x\big(x-c^3\big)\,{\bf v}\big)=\Big(q^{-3}a^{-1}\big(q^3-1\big)^{-1}\big(\big(aq^3-c^3\big)x
+c^3\big(1-aq^3\big)\big)\Big){\bf v}\;.$$
As a consequence, by Table \ref{Table1}, $q_n(x)=U_n\big(x; a,c^{-3}q^{-3}|q^3\big)$.
This confirms the result given on line (13) in Table \ref{Table3}.

\section*{Acknowledgements }
This work is partially supported by the Centre for Mathematics of the University of Coimbra -- UID/MAT/00324/2013, funded by FCT/MCTES and co-funded by the European Regional Development Fund through the Partnership Agreement PT2020, and by Project UID/Multi/04016/2016, funded by FCT/MCTES.
KC is supported by the Portuguese Government through the Funda\c{c}\~ao para a Ci\^encia e a Tecnologia (FCT) under the grant SFRH/BPD/101139/2014.
MNJ would like to thank the Instituto Polit\'ecnico de Viseu and CI\&DETS for their support. FM is partially supported by the Direcci\'on General de Investigaci\'on Cient\'ifica y T\'ecnica, Ministerio de Econom\'ia, Industria y Competitividad of Spain under the grant MTM2015--65888--C4--2--P. JP is partially supported by the Direcci\'on General de Investigaci\'on Cient\'ifica y T\'ecnica, Ministerio de Econom\'ia, Industria y Competitividad of Spain under the grant MTM2015--65888--C4--4--P.


\begin{thebibliography}{99}

\bibitem{Renato-libro} {R. \'Alvarez-Nodarse},
{\sl Polinomios hipergeom\'etricos y $q-$polin\'omios},
Monografias del Seminario Matem\'atico Garc\'\i a de Galdeano {\bf 26}, 2003.
(Revised edition: 2014.) In Spanish.

\bibitem{BarrucandDickinson} {P. Barrucand \and D. Dickinson},
{\it On cubic transformations of orthogonal polynomials},
Proc. Amer. Math. Soc. {\bf 17} (1966), 810-814.

\bibitem{KMP} {K. Castillo, M. N. de Jesus, \and J. Petronilho},
{\it  On semiclassical orthogonal polynomials via polynomial
mappings}, J. Math. Anal. Appl. {\bf 455} (2017), 1801-1821.

\bibitem{CharrisIsmail-siev2} {J. Charris \and M. E. H. Ismail},
{\it On sieved orthogonal polynomials II: random walk polynomials},
Canad. J. Math. {\bf 38} (1986), 397-414.

\bibitem{CharrisIsmail} {J. Charris \and M. E. H. Ismail},
{\it On sieved orthogonal polynomials VII: generalized polynomial
mappings}, Trans. Amer. Math. Soc. {\bf 340} (1993), 71-93.

\bibitem{CharrisIsmailMonsalve} {J. Charris, M. E. H. Ismail, \and S. Monsalve},
{\it On sieved orthogonal polynomials X: general blocks of recurrence relations},
Pacific J. Math. {\bf 163} (2) (1994), 237-267.

\bibitem{ChiharaBUMI} {T. S. Chihara},
{\it On kernel polynomials and related systems},
Boll. Un. Mat. Ital., Serie 3,
\textbf{19} (4) (1964), 451-459.

\bibitem{Chihara} {T. S. Chihara},
{\it An Introduction to  Orthogonal Polynomials} Gordon and Breach, New York, 1978.

\bibitem{LChihara} {L. M. Chihara \and T. S. Chihara}, A class of nonsymmetric orthogonal polynomials. J. Math. Anal. Appl.  \textbf{126}  (1987),  no. 1, 275–-291.

\bibitem{Hahn} {W. Hahn},
{\it  \"{U}ber orthogonal polynome, die $q-$differenzengleichungen
gen\"{u}gen}, Math. Nachr. {\bf 2} (1949), 4-34.

\bibitem{MarcioPetronilhoJAT} {M. N. de Jesus \and J. Petronilho},
{\it On orthogonal polynomials obtained via polynomial mappings},
J. Approx. Theory {\bf 162} (2010), 2243-2277.

\bibitem{JerVan} {J. Geronimo \and W. Van Assche},
{\it Orthogonal polynomials on several intervals via a polynomial mapping},
Trans. Amer. Math. Soc. {\bf 308} (1986), 559-581.

\bibitem{KM} {L. Kh\'eriji \and P. Maroni},
{\it  The $H_q-$classical orthogonal polynomials}, Acta Appl. Math. {\bf 71} (2002), 49-115.

\bibitem{Lotfi} {L. Kh\'eriji}:
{\it  An introduction to the $H_q-$semiclassical orthogonal
polynomials}, Methods Appl. Anal. {\bf 10} (3)
(2003), 387-411.

\bibitem{PacoZeLAA} {F. Marcell\'an \and J. Petronilho},
{\it Eigenproblems for tridiagonal $2-$Toeplitz matrices and quadratic polynomial mappings},
Linear Algebra Appl. {\bf 260} (1997), 169-208.

\bibitem{PacoZePortMath} {F. Marcell\'an \and J. Petronilho},
{\it Orthogonal polynomials and quadratic transformations},
Port. Math. {\bf 56}(1) (1999), 81-113.

\bibitem{PacoZeCubic1} {F. Marcell\'an \and J. Petronilho},
{\it Orthogonal polynomials and cubic polynomial mappings I},
Commun. Anal. Theory Contin. Fract. {\bf 8} (2000), 88-116.

\bibitem{PacoZeCubic2} {F. Marcell\'an \and J. Petronilho},
{\it Orthogonal polynomials and cubic polynomial mappings II:
the positive-definite case}, Commun. Anal. Theory Contin. Fract. {\bf 9} (2001), 11-20.

\bibitem{PacoGabriela} {F. Marcell\'an \and G. Sansigre},
{\it Orthogonal polynomials and cubic transformations},
J. Comput. Appl. Math. {\bf 49} (1993), 161-168.

\bibitem{Maroni} {P. Maroni},
{\it Une th\'eorie alg\'ebrique des polyn\^omes orthogonaux.
Applications aux polyn\^omes orthogonaux semiclassiques},
In C. Brezinski et al. Eds., Orthogonal Polynomials and Their Applications,
Proc. Erice 1990, IMACS, Ann. Comp. Appl. Math. {\bf 9} (1991), 95-130.

\bibitem{Maroni1} {P. Maroni}, {\it Sur la d\'ecomposition quadratique d'une suite de polyn\^{o}mes orthogonaux. I}.
Riv. Mat. Pura Appl.  \textbf{6} (1990), 19--53.

\bibitem{Maroni2} {P. Maroni}, {\it Sur la d\'ecomposition quadratique d'une suite de polyn\^{o}mes orthogonaux}.
II.  Portugal. Math. \textbf{50} (1993), no. 3, 305--329.

\bibitem{Mesquita}  {P. Maroni, T. A. Mesquita, \and Z. da Rocha}, {\it On the general cubic decomposition of polynomial sequences}. J. Difference Equ. Appl.  \textbf{17}  (2011),  no. 9, 1303–-1332.

\bibitem{MedemRenatoPaco} {J. C. Medem, R. \'Alvarez-Nodarse, \and F. Marcell\'an},
{\it On the $q-$polynomials: a distributional study},
J. Comput. Appl. Math. {\bf 135} (2001), 157-196.

\bibitem{Peherstorfer} {F. Peherstorfer},
{\it Inverse images of polynomial mappings and polynomials orthogonal on them},
J. Comput. Appl. Math. {\bf 153} (2003), 371–385.

\bibitem{Pettifor} {D. G. Pettifor \and  D. L. Weaire (editors),} {\it The recursion Method and its Applications} Springer Series uin Solid State Sciences \textbf{58}, Springer-Verlag, Berlin, 1985.

\bibitem{TounsiBouguerra} {M. I. Tounsi \and I. Bouguerra},
{\it Cubic decomposition of a family of semiclassical polynomial
sequences of class one}, Integral Transforms Spec. Funct. {\bf 26}
(2015) no. 5, 377-394.

\bibitem{TounsiRaddaoui} {M. I. Tounsi \and Z. Raddaoui},
{\it  On the cubic decomposition of a family of
$H_q-$semiclassical polynomial sequences of class one},
J. Differ. Equ. Appl. {\bf 22} (3) (2015), 391-410.

\bibitem{Wheeler} J. C. Wheeler, {\it Modified moments abd continued fraction coefficients for the diatomic linear chain}, J. Chem. Phys. \textbf{80} (1984), 472--476.

\end{thebibliography}
\end{document}